\documentclass[11pt]{amsart}
\usepackage[T1]{fontenc}
\usepackage[utf8]{inputenc}
\usepackage{amsmath,amssymb,amsthm,amsfonts,mathtools}
\usepackage{microtype}
\usepackage[margin=3cm]{geometry}
\usepackage[shortlabels]{enumitem}

\theoremstyle{plain}\newtheorem{theorem}{Theorem}
\theoremstyle{plain}\newtheorem{conjecture}[theorem]{Conjecture}
\theoremstyle{plain}\newtheorem{lemma}[theorem]{Lemma}
\theoremstyle{plain}
\theoremstyle{plain}
\theoremstyle{plain}\newtheorem{definition}[theorem]{Definition}
\theoremstyle{definition}\newtheorem{notation}[theorem]{Notation}
\numberwithin{theorem}{section}

\def\AD{\mathsf{AD}}
\def\ZF{\mathsf{ZF}}
\def\ZFC{\mathsf{ZFC}}
\def\DC{\mathsf{DC}}
\def\N{\mathbb{N}}
\def\R{\mathbb{R}}

\let\term\emph

\begin{document}

\title[Regressive functions on the hyperarithmetic degrees]{Martin's conjecture for regressive functions on the hyperarithmetic degrees}

\author{Patrick Lutz}
\address{Department of Mathematics, UC Berkeley}
\email{pglutz@berkeley.edu}

\begin{abstract}
We answer a question of Slaman and Steel by showing that a version of Martin's conjecture holds for all regressive functions on the hyperarithmetic degrees. A key step in our proof, which may have applications to other cases of Martin's conjecture, consists of showing that we can always reduce to the case of a continuous function. 
\end{abstract}

\maketitle

\section{Introduction}

Martin's conjecture is a proposed classification of the limit behavior of functions on the Turing degrees under strong set theoretic hypotheses (namely the Axiom of Determinacy). The full conjecture is still open, but several special cases have been proved. In particular, in \cite{slaman1988definable}, Slaman and Steel proved that Martin's conjecture holds for all ``regressive'' functions on the Turing degrees.

\begin{theorem}[$\ZF + \AD$; Slaman and Steel]
\label{thm-ss}
If $f \colon 2^\omega \to 2^\omega$ is a Turing-invariant function such that $f(x) \le_T x$ for all $x$ then either $f$ is constant on a cone or $f(x) \equiv_T x$ on a cone.
\end{theorem}

They also asked whether the analogous theorem for hyperarithmetic reducibility holds. In other words, is it possible to prove a version of Martin's conjecture for regressive functions on the hyperarithmetic degrees? Their motivation was as follows. A regressive function on the Turing degrees can be written as a countable union of continuous functions. Their argument works by using this fact to reduce to the case where $f$ is continuous and then showing that if such an $f$ is not constant on any cone then for all $x$ in some cone, it is possible to find $y \equiv_T x$ such that $x$ is coded into $f(y)$.
In their coding argument, they relied strongly on the properties of continuous functions. In contrast with regressive functions on the Turing degrees, regressive functions on the hyperarithmetic degrees can only be written as countable unions of Borel functions. Thus Martin's conjecture for regressive functions on the hyperarithmetic degrees forms a natural test case to see whether their coding argument can be extended to deal with functions which are not continuous.

The main result of this paper is to answer their question in the affirmative. Namely we will prove the following theorem.

\begin{theorem}[$\ZF + \AD$]
\label{thm-main}
Let $f \colon 2^\omega \to 2^\omega$ be a hyp-invariant function such that $f(x) \le_H x$ for all $x$. Then either $f$ is constant on a cone of hyperarithmetic degrees or $f(x) \equiv_H x$ on a cone of hyperarithmetic degrees.
\end{theorem}

There are a few interesting things to note about our proof. First, instead of adapting Slaman and Steel's methods to work with non-continuous functions, we instead show that $f$---despite potentially being far from continuous---can be replaced by a hyp-equivalent function which \emph{is} continuous. We still have to modify their coding argument to work with hyperarithmetic reducibility rather than Turing reducibility, but in doing so we make heavy use of the fact that we can assume we are dealing with a continuous function.

This suggests that in some cases of Martin's conjecture where the functions being considered are not continuous, it may still be possible to replace them with related functions which are continuous. This idea has already borne fruit in the form of \cite{lutz2023part}, where it is combined with a refined version of the coding arguments introduced in this paper to prove part 1 of Martin's conjecture for order-preserving functions. 

Second, our results cast at least a little doubt on the idea that any use of determinacy in proving Martin's conjecture will be ``local'' (that is, the idea that only Borel determinacy is needed when dealing with Borel functions, and so on). Our proof seems to use more than Borel determinacy, even when the functions being considered are assumed to be Borel (specifically, our proof uses analytic determinacy). In section \ref{section-borel}, we show that Borel determinacy \emph{is} sufficient, but this requires a more careful analysis that was not needed for the $\AD$ proof.

Third, our reduction to the case of a continuous function is quite flexible and seems to work in many different degree structures, including the arithmetic degrees. Somewhat surprisingly, it seems much harder to adapt the coding argument used by Slaman and Steel, even once we are allowed to assume we are dealing with a continuous function. In this paper, we have to use a somewhat different coding argument than the one used by Slaman and Steel, and in doing so we have to rely on the $\Sigma^1_1$-bounding theorem. Also, we have so far not been able to modify either our coding argument or Slaman and Steel's to work for arithmetic reducibility (in our opinion, the regressive case of Martin's conjecture on the arithmetic degrees is an interesting open question).

\subsection*{Acknowledgements} Thanks to Ted Slaman, Benny Siskind, Gabe Goldberg, Vittorio Bard and James Walsh for helpful conversations and advice. The author would like to especially thank Ted Slaman for mentioning the problem and encouraging him to work on it.

\section{Preliminaries}

In this section we will provide some background on hyperarithmetic reducibility and on Martin's conjecture and then state some lemmas that we will use in the proof of Theorem \ref{thm-main}. All the lemmas are standard, with the exception of Lemma \ref{lemma-computableuniformization}, which we will see has a simple proof using standard techniques. For the reader intimidated by the axiom of determinacy, we note that the only way we will use determinacy in this paper is in the form of Lemma \ref{lemma-computableuniformization}.

\subsection{Background on Hyperarithmetic Reducibility}

The easiest definition of \term{hyperarithmetic reducibility} is that $y \le_H x$ if $y$ is $\Delta^1_1(x)$ definable (in which case we will often say that $y$ is \term{hyperarithmetic} in $x$). It is not very hard to see that this relation is transitive and thus deserves the title ``reducibility.'' As usual, we can then define hyperarithmetic equivalence and the structure of the hyperarithmetic degrees.

But there is another characterization of hyperarithmetic reducibility which is often useful and which we will now explain. Let $\omega_1^x$ denote the least countable ordinal with no presentation computable from $x$. Work of Davis, Kleene and Spector shows that for any $\alpha < \omega_1^x$, there is a notion of the $\alpha^\text{th}$ iterate of the jump of $x$ which is well-defined up to Turing equivalence~\cite{davis1950relatively, kleene1955hierarchies, spector1955recursive}. We denote this $\alpha^{\text{th}}$ jump of $x$ by $x^{(\alpha)}$. Kleene proved in~\cite{kleene1955hierarchies} that $y$ is hyperarithmetic in $x$ if and only if $y \leq_T x^{(\alpha)}$ for some $\alpha < \omega_1^x$. 

It will be helpful later in the paper if we make some of this more precise. Suppose $r$ is a real which codes a linear order $\leq_r$ on $\N$ which has a minimum element, $0_r$. If $x$ is any real, then a jump hierarchy on $r$ which starts with $x$ is a set $H \subset \N^2$ such that the $0_r^\text{th}$ column of $H$ is $x$ and for each $n \ne 0_r$, the $n^\text{th}$ column of $H$ is equal to the jump of the smaller columns of $H$ (smaller according to the ordering given by $\leq_r$). In other words, if we define
\begin{align*}
H_n &= \{i \mid \langle n, i\rangle \in H\}\\
H_{< n} &= \{\langle m, i\rangle \mid m <_r n \text{ and } \langle m, i\rangle \in H\}
\end{align*}
then we have $H_{0_r} = x$ and $H_n = (H_{< n})'$ for all $n \ne 0$. 

If $\leq_r$ happens to be a presentation of a well-order then there is always a unique $H$ satisfying the conditions above. Moreover, if $\alpha < \omega_1^x$ and $r$ codes a presentation of $\alpha$ which is computable from $x$ then the Turing degree of the unique jump hierarchy on $r$ starting from $x$ is independent of the specific choice of $r$. Such a jump hierarchy is considered to be the $\alpha^\text{th}$ jump of $x$ (which is only well-defined up to Turing degree). This makes precise the alternative characterization of hyperarithmetic reducibility mentioned above.

It is also worth mentioning here that hyperarithmetic reducibility is closely connected to Borel measurability. Just as every continuous function is computable relative to some oracle, every Borel function is hyperarithmetic relative to some oracle. More precisely, if $f$ is Borel then there is some countable ordinal $\alpha$, some $r$ which codes a presentation of $\alpha$, some real $y$ and some Turing functional $\Phi$ such that for all $x$, $f(x) = \Phi((x \oplus y)^{(\alpha)})$, where $(x\oplus y)^{(\alpha)}$ is taken to mean the unique jump hierarchy on $r$ starting from $x \oplus y$.

\subsection{Background on Martin's Conjecture}

As mentioned in the introduction, Martin's conjecture is a proposed classification of the limit behavior of functions on the Turing degrees under strong set theoretic hypotheses. It is traditionally divided into two parts. We will only discuss the first part here, since that is all that is relevant for this paper.

Very roughly, part 1 of Martin's conjecture states that if $f$ is a function from the Turing degrees to the Turing degrees then either $f(x)$ is constant for all large enough $x$ or $f(x) \ge_T x$ for all large enough $x$. There are three things to explain here. First, a caveat: the conjecture is actually stated not in terms of functions on the Turing degrees, but in terms of Turing invariant functions on the reals. Second, we need to state precisely what ``for all large enough $x$'' really means. Third, the conjecture is false in $\ZFC$ and is instead stated as a conjecture in the theory $\ZF + \AD$ (or sometimes $\ZF + \AD + \DC_{\R}$, though we will not need to use $\DC_{\R}$ in this paper). We will now explain each of these points in more detail.

First, let's define precisely what we mean by a \term{Turing invariant function on the reals}. A function $f \colon 2^\omega \to 2^\omega$ is called \term{Turing invariant} if for all $x$ and $y$ in $2^\omega$,
\[
x \equiv_T y \implies f(x) \equiv_T f(y).
\]
The point is that a Turing invariant function $f$ induces a function on the Turing degrees. Using the Axiom of Choice, it is clear that every function on the Turing degrees arises from a Turing invariant function on the reals, but this may fail in $\ZF$ (though it is true again if we assume $\AD_{\R}$, a strengthening of the Axiom of Determinacy). So Martin's conjecture is actually only classifying the behavior of functions on the Turing degrees which come from Turing invariant functions on the reals.

Since it will be useful to us, we will also mention here the definition of a \term{Turing invariant set of reals}. A subset $A \subseteq 2^\omega$ is called \term{Turing invariant} if for all $x$ and $y$ in $2^\omega$,
\[
x \equiv_T y \implies (x \in A \leftrightarrow y \in A).
\]

Next, let's explain what we mean by ``all large enough $x$.'' The key concept is that of a \term{cone of Turing degrees} (which is actually a Turing invariant subset of $2^\omega$ rather than a subset of the Turing degrees): a cone of Turing degrees is a set of the form $\{x \in 2^\omega \mid x \ge_T y\}$ for some fixed $y$. This $y$ is called the \term{base} of the cone and the cone is sometimes referred to as the \term{cone above $y$}. What we mean by ``all large enough $x$'' is simply ``for all $x$ in some cone.''

Third, we will mention a few things about the Axiom of Determinacy. The Axiom of Determinacy (often written $\AD$) is an axiom of set theory which is inconsistent with the Axiom of Choice and equiconsistent with the existence of infinitely many Woodin cardinals~\cite{koellner2009large}. We will not give a definition of the Axiom of Determinacy here, but simply mention the following fact, which is one of the main consequences of $\AD$ for computability theory.

\begin{theorem}[$\ZF + \AD$; Martin]
\label{thm-cone}
If $A$ is a Turing invariant subset of $2^\omega$ then either $A$ contains a cone or $A$ is disjoint from a cone.
\end{theorem}

There is also a weak form of determinacy called ``Borel determinacy'' which \emph{is} provable in $\ZF$ and which is enough to prove Theorem \ref{thm-cone} if the set $A$ is assumed to be Borel.

We can now give a formal statement of part 1 of Martin's conjecture.

\begin{conjecture}[Part 1 of Martin's Conjecture]
Assuming $\ZF + \AD$, if $f \colon 2^\omega \to 2^\omega$ is a Turing invariant function then either $f(x) \ge_T x$ for all $x$ in some cone or there is some $y$ such that $f(x) \equiv_T y$ for all $x$ in some cone.
\end{conjecture}

In a slight abuse of terminology, the latter possibility in the conjecture is often written as ``$f$ is constant on a cone'' (even though it is the function that $f$ induces on the Turing degrees that is constant, not $f$ itself).

Finally, we mention that for many degree structures besides the Turing degrees (and in particular for the hyperarithmetic degrees), it is possible to state a sensible version of Martin's conjecture by just swapping out Turing reducibility for the appropriate alternative notion of reducibility in the definitions of ``Turing invariant function'' and ``cone of Turing degrees.'' This is reasonable to do in part because Theorem \ref{thm-cone} works for pretty much any notion of reducibility stronger than Turing reducibility (and also for many which are weaker).

\subsection{Determinacy Lemmas}

We now state a few lemmas that will help us apply determinacy even in situations where we have to deal with non-Turing invariant sets of reals. The key notion is that of a ``pointed perfect tree.''

\begin{definition}
A \term{perfect tree} is a tree, $T$, such that every node in $T$ has a pair of incompatible extensions which are both in $T$.
\end{definition}

\begin{definition}
A \term{pointed perfect tree} is a perfect tree, $T$, such that every path through $T$ computes $T$.
\end{definition}

\begin{notation}
If $T$ is a tree, we will use $[T]$ to refer to the set of paths through $T$.
\end{notation}

The reason pointed perfect trees are useful to work with is that if $T$ is a pointed perfect tree then $[T]$ contains a representative of every Turing degree which is above the Turing degree of $T$. Next, we will see that determinacy can be used to get pointed perfect trees. For a proof of Lemma \ref{lemma-determinacy}, see Lemma 3.5 of \cite{marks2016martins}.\footnote{Note that, as stated, Lemma 3.5 of \cite{marks2016martins} does not quite match our Lemma~\ref{lemma-determinacy} below. In particular, it is stated under the assumption of $\ZF + \AD + \DC$ and the function $h$ is defined on all of $2^\omega$ instead of just on a set which is cofinal in the Turing degrees. However, the proof never uses $\DC$ and is easy to modify to work with a function only defined on a cofinal set.}

\begin{definition}
A set $A \subseteq 2^\omega$ is cofinal in the Turing degrees if for all $x$ there is some $y \ge_T x$ such that $y \in A$ (note that $A$ is not required to be Turing invariant).
\end{definition}

\begin{lemma}[$\ZF + \AD$]
\label{lemma-determinacy}
Suppose $A \subseteq 2^\omega$ is cofinal in the Turing degrees and $h$ is a function on $A$ with countable range. Then there is a pointed perfect tree on which $h$ is constant.
\end{lemma}

The following lemma will be our only use of determinacy in the proofs in the rest of this paper. Essentially it is a kind of computable uniformization principle provable from $\AD$.

\begin{lemma}[$\ZF + \AD$]
\label{lemma-computableuniformization}
Suppose $R$ is a binary relation on $2^\omega$ such that
\begin{itemize}
    \item The domain of $R$ is cofinal in the Turing degrees: for all $z$ there is some $x \ge_T z$ and some $y$ such that $(x, y) \in R$
    \item and $R$ is a subset of Turing reducibility: for every $(x, y) \in R$, $x \ge_T y$.
\end{itemize}
Then there is a pointed perfect tree $T$ and a Turing functional $\Phi$ such that for every $x \in [T]$, $\Phi(x)$ is total and $(x, \Phi(x)) \in R$. In other words, $\Phi$ is a computable choice function for $R$ on $[T]$.
\end{lemma}

\begin{proof}
For each $x$ in the domain of $R$ there is some $e$ such that $\Phi_e(x)$ is total and $R(x, \Phi_e(x))$ holds. Let $e_x$ denote the smallest such $e$. By determinacy (in the form of Lemma \ref{lemma-determinacy}), there is a pointed perfect tree $T$ on which $e_x$ is constant. Let $e$ be this constant value. Then $T$ and $\Phi_e$ satisfy the conclusion of the lemma.
\end{proof}

It will be useful below to note that if $A$ and $h$ in Lemma \ref{lemma-determinacy} are Borel then the result is provable in $\ZF$, and similarly that if $R$ in the above lemma is assumed to be Borel, then that result, too, is provable in $\ZF$.

\subsection{Pointed Perfect Tree Lemmas}

Now we will state a couple of lemmas that are helpful when working with pointed perfect trees. These lemmas do not require the Axiom of Determinacy. The proofs are routine, so we omit them, but note that Lemma~\ref{lemma-thinning} is essentially Lemma V.2.7 of~\cite{lerman1983degrees} and Lemma~\ref{lemma-invertible} is similar to Lemma V.2.6 (though with an extra use of compactness necessary). Proofs of both lemmas can also be found in~\cite{lutz2023part} (see Lemmas 2.7 and 2.1).

\begin{lemma}
\label{lemma-thinning}
Suppose $T$ is a pointed perfect tree and $\Phi$ is a Turing functional such that $\Phi(x)$ is total for every $x \in [T]$. Then either $\Phi$ is constant on a pointed perfect subtree of $T$ or $\Phi$ is injective on a pointed perfect subtree of $T$.
\end{lemma}

\begin{lemma}
\label{lemma-invertible}
Suppose $T$ is a perfect tree and $\Phi$ is a Turing functional such that $\Phi(x)$ is total for every $x \in [T]$ and $\Phi$ is injective on $[T]$. Then for each $x \in [T]$,
\[
\Phi(x) \oplus T \ge_T x.
\]
In fact, this reduction is even uniform in $x$ (though we won't need to use that fact in this paper).
\end{lemma}

\subsection{Computable Linear Orders}

To work with hyperarithmetic reducibility, we will need to make use of a few facts about computable linear orders and computable well-orders. Proofs can be found in \cite{sacks1990higher}.

One of the most important facts about computable well-orders is the $\Sigma^1_1$-bounding theorem. Essentially it says that every $\Sigma^1_1$-definable collection of well-orders is bounded below a computable ordinal. The theorem comes in multiple flavors, depending on whether we are talking about sets of programs which compute presentations of well-orders, or real numbers which \emph{are} presentations of well-orders and depending on whether the $\Sigma^1_1$ definition is boldface, lightface, or lightface relative to some fixed real. Below, we just state the two versions that we will need in this paper.

\begin{theorem}
Suppose that $x$ is a real and $A$ is a $\Sigma^1_1(x)$ definable set of codes for programs such that for every $e$ in $A$, $\Phi_e(x)$ is a presentation of a well-order. Then there is some $\alpha < \omega_1^x$ which is greater than every ordinal with a presentation coded by an element of $A$.
\end{theorem}

\begin{theorem}
If $A$ is a $\mathbf{\Sigma^1_1}$ definable set of presentations of well-orders then there is some $\alpha < \omega_1$ which is greater than every ordinal with a presentation in $A$.
\end{theorem}

We will also need some ideas originally introduced by Harrison in \cite{harrison1968recursive}.

\begin{definition}
If $x$ is a real and $r$ is a real computable from $x$ that codes a presentation of a linear order, then $r$ is a \term{pseudo-well-order relative to $x$} if it is ill-founded but contains no infinite descending sequence which is hyperarithmetic in $x$.
\end{definition}

\begin{lemma}
\label{lemma-harrisondefinable}
If $r$ is a presentation of a linear order that is computable from a real $x$ then the assertion ``$r$ has no infinite descending sequence which is hyperarithmetic in $x$'' is equivalent to a $\Sigma^1_1(x)$ formula.
\end{lemma}

\begin{lemma}
\label{lemma-pseudowellordercomputes}
If $r$ is a pseudo-well-order relative to $x$ and $H$ is a jump hierarchy on $r$ that starts with $x$ then $H$ computes every real which is hyperarithmetic in $x$.
\end{lemma}

\section{Proof of the Main Theorem}

In this section, we will prove Theorem \ref{thm-main}. Before we launch into the details of the proof, we will give an outline of the general strategy. And before we do that, we will recall the general strategy followed by Slaman and Steel in their proof of Theorem \ref{thm-ss}. The steps of their proof are essentially as follows.
\begin{itemize}
    \item First, use determinacy to show that there is a pointed perfect tree on which $f$ is computable. Then use Lemma \ref{lemma-thinning} to show that we can also assume $f$ is injective.
    \item Next, show that if $x$ is in the pointed perfect tree, then every function computable from $x$ is dominated by a function computable from $f(x)$. The idea is that if $x$ computes a function which is not dominated by any function computable from $f(x)$ then $x$ can diagonalize against $f(x)$ by using this function to guess convergence times for $f(x)$ programs. The diagonalization produces a real $y$ in the same Turing degree as $x$ such that $f(x)$ cannot compute $f(y)$, thereby contradicting the Turing invariance of $f$.
    \item Once you can assume that $f$ is computable and injective on a pointed perfect tree and that if $x$ is in this tree then every function computed by $x$ is dominated by a function computed by $f(x)$, use a coding argument to show that $f(x) \ge_T x$. The coding argument works by coding bits of $x$ into the relative growth rates of two fast growing functions computed by $f(x)$.
\end{itemize}
Our proof makes three main modifications to this outline. First, instead of showing that $f$ is computable on some pointed perfect tree, we show that $f$ is hyp-equivalent to some computable function on a pointed perfect tree. Thus we may work with that function instead of $f$. Second, instead of showing that every fast growing function computed by $x$ is dominated by a function computed by $f(x)$, we show that every well-order computed by $x$ embeds into a well-order computed by $f(x)$---in other words that $\omega_1^x = \omega_1^{f(x)}$. Third, instead of coding bits of $x$ into the relative growth rates of fast growing functions computed by $f(x)$, we code the bits of $x$ into the Kolmogorov complexities of initial segments of reals computed by $f(x)$ (though it is not necessary to know anything about Kolmogorov complexity to follow our argument). Also, to be able to carry out the coding argument, we will first have to use a trick involving $\Sigma^1_1$-bounding. To sum up, here's an outline of our proof.
\begin{itemize}
    \item First, we will use determinacy to replace $f$ with a hyp-equivalent function which is computable on a pointed perfect tree. By using Lemma \ref{lemma-thinning}, we can also assume that $f$ is injective.
    \item Next we show that $\omega_1^{f(x)} = \omega_1^x$ for all $x$ in the pointed perfect tree. The idea is if $\omega_1^{f(x)}$ was less than $\omega_1^x$ then $x$ would be able to diagonalize against $f(x)$ by using $\omega_1^{f(x)}$ jumps.
    \item Once we are able to assume that $f$ is computable and injective and that $\omega_1^{f(x)} = \omega_1^x$, we will use a coding argument to show that $f(x) \ge_H x$. In our coding argument, it will be important to know that there is a single ordinal $\alpha < \omega_1^x$ such that for every real $y$ in the same Turing degree as $x$, $f(x)^{(\alpha)}$ computes $f(y)$. We will prove this fact using $\Sigma^1_1$-bounding.
\end{itemize}
And now it's time to present the actual proof.

\subsection{Replacing $f$ with an injective, computable function}

First we will show that $f$ can be replaced by a computable function. This is the only part of the proof that uses determinacy.

\begin{lemma}[$\ZF + \AD$]
\label{lemma-replacef}
Suppose $f \colon 2^\omega \to 2^\omega$ is hyp-invariant and hyp-regressive. Then there is a Turing functional $\Phi$ and a pointed perfect tree $T$ such that for all $x \in [T]$, $\Phi(x)$ is total and $\Phi(x) \equiv_H f(x)$.
\end{lemma}

\begin{proof}
Consider the following binary relation, $R$.
\[
  R(x, y) \iff x \geq_T y \text{ and } f(x) \equiv_H y.  
\]
The idea is that a computable function which is hyp-equivalent to $f$ is exactly a computable function which uniformizes $R$. To show that such a function exists, it suffices to check that we can apply Lemma \ref{lemma-computableuniformization}. 

To check that we can apply Lemma \ref{lemma-computableuniformization}, we need to check that the domain of $R$ is cofinal in the Turing degrees and that $R$ is a subset of Turing reducibility. The latter is clear from the definition of $R$. For the former, fix any real $x$ and we will show that some real which computes $x$ is in the domain of $R$. Since $f(x) \leq_H x$, there is some $\alpha < \omega_1^x$ such that $x^{(\alpha)} \geq_T f(x)$. Since $x^{(\alpha)} \equiv_H x$ and $f$ is hyp-invariant, $f(x^{(\alpha)}) \equiv_H f(x)$. Thus $R(x^{(\alpha)}, f(x))$ holds and so $x^{(\alpha)}$ is an element of the domain of $R$ which computes $x$.

Thus we may apply Lemma \ref{lemma-computableuniformization} to get a pointed perfect tree $T$ and a Turing functional $\Phi$ such that for all $x \in [T]$, $\Phi(x)$ is total and $\Phi(x) \equiv_H f(x)$.
\end{proof}

For the rest of the proof we will simply assume that $f$ is computable on a pointed perfect tree. It will also be convenient to assume that $f$ is injective on a pointed perfect tree, which we show next.

\begin{lemma}[$\ZF$]
Suppose $T$ is a pointed perfect tree and $f \colon 2^\omega \to 2^\omega$ is a hyp-invariant function which is computable on $[T]$. Then either $f$ is constant on a cone of hyperdegrees or $f$ is injective on a pointed perfect subtree of $T$.
\end{lemma}

\begin{proof}
By Lemma \ref{lemma-thinning}, either $f$ is constant on a pointed perfect subtree of $T$ or $f$ is injective on a pointed perfect subtree of $T$. In the former case, $f$ is constant on a cone of hyperdegrees and in the latter case, we are done.
\end{proof}

For the rest of the proof, we will deal with the case of a hyp-invariant function, $f$, which is computable and injective on a pointed perfect tree, $T$. We will show that for any $x$ in $[T]$, $f(x) \ge_H x$. There are two cases: when $\omega_1^{f(x)} < \omega_1^x$ and when $\omega_1^{f(x)} = \omega_1^x$. We will show that the first case is impossible and that if we are in the second case then we can use the coding argument mentioned above.

\subsection{Proving that $f$ preserves $\omega_1^x$}

We will now show that for any $x \in [T]$, $\omega_1^{f(x)} = \omega_1^x$. We will do this by deriving a contradiction from the assumption that $\omega_1^{f(x)} < \omega_1^x$ (note that since $f(x) \leq_H x$ we cannot have $\omega_1^{f(x)} > \omega_1^x$). The basic idea is that in this case we can diagonalize against $f(x)$. Namely, we can use $\omega_1^{f(x)}$ jumps of $x$ to compute a real $y$ so that $f(x)$ cannot compute $f(y)$ with fewer than $\omega_1^{f(x)}$ jumps (and hence $f(x)$ cannot be hyp-equivalent to $f(y)$). Since $\omega_1^x > \omega_1^{f(x)}$, this $y$ can be made hyp-equivalent to $x$, which violates the hyp-invariance of $f$. We now give the formal proof.

\begin{lemma}[$\ZF$]
\label{lemma-case1}
Suppose $T$ is a pointed perfect tree and $f$ is a hyp-invariant function which is computable and injective on $[T]$. Then for every $x \in [T]$, $\omega_1^{f(x)} = \omega_1^x$.
\end{lemma}

\begin{proof}
Suppose for contradiction that for some $x \in [T]$, $\omega_1^{f(x)} < \omega_1^x$. Let $\alpha = \omega_1^{f(x)}$. The key point is that for every $y \in [T]$ which is hyp-equivalent to $x$, $x^{(\alpha)}$ computes $y$.

Why is that? Well, if $y$ is in the same hyperdegree as $x$ then $f(y)$ is in the same hyperdegree as $f(x)$. So by definition of $\alpha$, there is some $\beta < \alpha$ such that $f(x)^{(\beta)} \ge_T f(y)$. We then have the following calculation.
\begin{align*}
x^{(\alpha)} &\ge_T x^{(\beta)} &\text{because $\beta < \alpha$}\\
&\ge_T x^{(\beta)} \oplus T &\text{because $T$ is pointed}\\
&\ge_T f(x)^{(\beta)} \oplus T &\text{because $f(x) \le_T x$}\\
&\ge_T f(y)\oplus T &\text{by definition of $\beta$}\\
&\ge_T y &\text{by Lemma \ref{lemma-invertible}.}
\end{align*}

We can now finish the proof easily. Since $T$ is pointed, we can pick some $y \in [T]$ which is Turing equivalent to $x^{(\alpha + 1)}$. Since $\alpha < \omega_1^x$, this $y$ is hyp-equivalent to $x$. But it obviously is not computable from $x^{(\alpha)}$, so we have reached a contradiction.
\end{proof}

\subsection{Coding argument}

In this part of the proof, we will explain how to code $x$ into some real of the same hyperarithmetic degree as $f(x)$. The argument has some similarity to the proof of a basis theorem for perfect sets given by Groszek and Slaman in \cite{groszek1998basis} (which itself has some similarity to the coding argument used in \cite{slaman1988definable}). Before giving the coding argument, however, we will first show that for every $x \in [T]$ there is a uniform bound on the number of jumps that $f(x)$ takes to compute $f(y)$ for any $y \in [T]$ which is Turing equivalent to $x$.

\begin{lemma}[$\ZF$]
\label{lemma-bound}
Suppose $T$ is a pointed perfect tree, $f$ is a hyp-invariant function which is computable on $[T]$, and $x \in [T]$. Then there is some $\alpha < \omega_1^x$ such that if $y \in [T]$ is Turing equivalent to $x$ then $f(y) \le_T f(x)^{(\alpha)}$.
\end{lemma}

\begin{proof}
The main idea is just to use $\Sigma^1_1$-bounding. Let $A$ be the set of programs $e$ such that $\Phi_e(f(x))$ computes a linear order $r$ for which
\begin{enumerate}
    \item $r$ has no infinite descending sequence which is hyperarithmetic in $f(x)$
    \item and there is some $y \equiv_T x$ in $[T]$ and some jump hierarchy $H$ on $r$ starting from $f(x)$ such that $H$ does not compute $f(y)$.
\end{enumerate}
By Lemma \ref{lemma-harrisondefinable} (and since $f(x)$ is computable from $x$), $A$ is $\Sigma^1_1(x)$. I claim that every program in $A$ computes a well-order.

Suppose instead that $A$ contains a program $e$ computing an ill-founded order, $r$. Thus $r$ is a pseudo-well-order relative to $f(x)$. Since $e$ is in $A$, there must be some $y\equiv_T x$ in $[T]$ and some jump hierarchy on $r$ starting with $f(x)$ which does not compute $f(y)$. And since $f$ is hyp-invariant, we must have $f(x) \equiv_H f(y)$. But by Lemma \ref{lemma-pseudowellordercomputes}, any jump hierarchy on $r$ which starts with $f(x)$ computes everything in the hyperdegree of $f(x)$, and in particular $f(y)$. This is a contradiction, so all programs in $A$ must compute well-orders. 

Since $A$ is $\Sigma^1_1(x)$ and contains only programs computing well-orders, $\Sigma^1_1$-bounding implies that there is some $\alpha < \omega_1^x$ which bounds every well-order in $A$. This implies that for every $y \equiv_T x$ in $[T]$, $f(y)$ is computable from $f(x)^{(\alpha + 1)}$.
\end{proof}

We now come to the coding argument. As we have discussed, it replaces a different coding argument used by Slaman and Steel, and while their argument codes information into the relative growth rates of two fast-growing functions, ours codes information into the relative Kolmogorov complexities of initial segments of three reals (though the reader does not need to be familiar with Kolmogorov complexity to understand the proof below).

\begin{lemma}[$\ZF$]
\label{lemma-coding}
Suppose $T$ is a pointed perfect tree and $f$ is a hyp-invariant function which is computable and injective on $[T]$. Then $f(x) \ge_H x$ for all $x \in [T]$.
\end{lemma}

\begin{proof}
Let $x \in [T]$. Our goal is to show that $f(x) \ge_H x$. By Lemma \ref{lemma-bound}, there is some $\alpha < \omega_1^x$ such that for all $y \in [T]$ in the same Turing degree as $x$, we have $f(x)^{(\alpha)} \ge_T f(y)$. By Lemma \ref{lemma-case1}, we know that $\omega_1^x = \omega_1^{f(x)}$ and thus $\alpha < \omega_1^{f(x)}$. By thinning $T$, we may assume that $x$ is the base of $T$---i.e.\ for every $y \in [T]$, $x \leq_H y$. This assumption is useful because it means that if we construct an element $y \in [T]$ which is hyperarithmetic in $x$ then we can be sure that $x$ and $y$ are actually hyp-equivalent. We will use this below without further comment.

To complete the proof, we will find reals $a, b, c \in [T]$ which are hyp-equivalent to $x$ such that $x \le_T f(x)^{(\alpha + 2)} \oplus f(a) \oplus f(b) \oplus f(c)$. To see why this is sufficient, simply note that by hyp-invariance of $f$ and the fact that $\alpha < \omega_1^{f(x)}$, $f(x)$ is hyp-equivalent to $f(x)^{(\alpha + 2)} \oplus f(a) \oplus f(b) \oplus f(c)$.

We will construct $a, b,$ and $c$ in stages, using the method of finite extensions. That is, on stage $n$ of the construction we will construct finite initial segments $A_n, B_n,$ and $C_n$ of $a, b,$ and $c$, respectively. 

To ensure that the reals $a, b,$ and $c$ are in $[T]$, we will make sure that on each step, each of the initial segments we have constructed so far are also initial segments of \emph{some} path through $T$. More precisely, we will ensure that on stage $n$, there are reals $a_n, b_n,$ and $c_n$ in $[T]$ such that $A_n$ is an initial segment of $a_n$, and similarly for $B_n$ and $C_n$. Thus the reals $a, b,$ and $c$ we construct are each a limit of elements of $[T]$ and hence themselves elements of $[T]$. For reasons that will be explained more fully below, we will also require that $a_n, b_n,$ and $c_n$ are all computable from $x$. We will refer to $a_n, b_n,$ and $c_n$ as the \term{stage $n$ targets} for $a, b,$ and $c$.

To ensure that $a, b,$ and $c$ are all hyp-equivalent to $x$, we will make sure that the entire construction is hyperarithmetic from $x$. This ensures that $a, b, c \leq_H x$; since each of $a, b,$ and $c$ is in $[T]$, they are all hyp above $x$ and hence all hyp-equivalent to $x$.

To help explain the construction, we will use a metaphor. Imagine there are two people, who we will call the \term{coder} and the \term{decoder}, locked in separate rooms. The coder has access to $x$ and needs to communicate it to the decoder. In order to do so, the coder can construct $a, b,$ and $c$ (in a hyperarithmetic way). The decoder, who has access to $f(x)^{(\alpha + 2)}$, is then give $f(a), f(b),$ and $f(c)$ and is tasked with recovering $x$ (in a computable way, though even doing it in a hyperarithmetic way would be enough). In order to help the decoder do this, the coder needs to not just encode the bits of $x$, but also enough information to help the decoder figure out how to carry out the decoding process.

\medskip\noindent\textbf{Overview.} Before describing the coding and decoding processes in more detail, we will give an overview of what happens on each step. Suppose we are at stage $n$. By the end of the stage, the coder must pick the following.
\begin{itemize}
    \item Finite strings $A_n, B_n,$ and $C_n$ in $T$ which will be initial segments of $a, b,$ and $c$. These should extend whatever strings the coder picked on the previous step.
    \item Reals $a_n, b_n,$ and $c_n$ in $[T]$ which are all Turing equivalent to $x$ and such that $A_n$ is an initial segment of $a_n$ and similarly for $B_n$ and $C_n$.
\end{itemize}
The coder will also keep track of indices for programs $e_{a, n}, e_{b, n},$ and $e_{c, n}$ where $e_{a, n}$ is the least index of a program computing $f(a_n)$ from $f(x)^{(\alpha)}$ and similarly for $e_{b, n}$ and $e_{c, n}$. Recall that we stated that $a_n, b_n,$ and $c_n$ are Turing equivalent to $x$ and thus by our choice of $\alpha$, $f(a_n), f(b_n),$ and $f(c_n)$ are all computable from $f(x)^{(\alpha)}$. Therefore $e_{a, n}, e_{b, n},$ and $e_{c, n}$ are all well-defined. The idea is that the $n^\text{th}$ bit of $x$ will be encoded into these indices.

Actually, on each step of the coding process, the coder will encode the next bit of $x$ using only two of $a, b,$ and $c$ and the third will play a ``helper'' role, coding some auxiliary information to help the decoder carry on the decoding process. Which of $a, b,$ and $c$ is playing this helper role will simply rotate between them on each step. So, for example, $a$ will play the helper role on every third step.

Suppose that $c$ is playing the helper role on stage $n$. Then the coder will encode the $n^\text{th}$ bit of $x$ into the relative sizes of $e_{a, n}$ and $e_{b, n}$. In particular, if the $n^\text{th}$ bit is $0$ then the coder will arrange that $e_{a, n} > e_{b, n}$ and if the $n^\text{th}$ bit is $1$ then the coder will arrange that $e_{a, n} < e_{b, n}$. In order to help the decoder determine $e_{a, n}$ and $e_{b, n}$, the coder will use their ability to control initial segments of $f(a), f(b),$ and $f(c)$, which they can do because $f$ is continuous on $[T]$ and they can control initial segments of $a, b,$ and $c$.

Still supposing that $c$ is playing the helper role on stage $n$, let's consider what the decoder should do. In order to decode the next bit of $x$, the decoder must determine $e_{a, n}$ and $e_{b, n}$ (they do not need to find $e_{c, n}$). To do so, they will use $e_{c, n - 1}$ (which is what we mean when we say that $c$ plays a helper role on this step). Thus it is essential that the decoder has correctly determined $e_{c, n - 1}$, which is why the helper role rotates between $a, b,$ and $c$. In other words, on every step, one of the two indices that the decoder has correctly determined at the end of the step will be used to help the decoder carry out the subsequent step.

\medskip\noindent\textbf{Decoding process.} 
We will now explain the coding and decoding processes in more detail. We will begin with the decoding process. Suppose that the decoder has finished stage $n$ and is at the beginning of stage $n + 1$. For the sake of concreteness, we will assume that $c$ is playing the helper role on stage $n + 1$. We will also make the inductive assumption that the decoder correctly determined $e_{c, n}$ on stage $n$. The decoder now does the following.
\begin{enumerate}
    \item First, the decoder looks for the first place where $f(c)$ disagrees from $\Phi_{e_{c, n}}(f(x)^{(\alpha)})$, i.e. the least $m$ such that $f(c)$ and $\Phi_{e_{c, n}}(f(x)^{(\alpha)})$ disagree on their $m^\text{th}$ bit but agree on all earlier bits. Recall that since we are assuming the decoder has correctly determined $e_{c, n}$, $\Phi_{e_{c, n}}(f(x)^{(\alpha)}) = f(c_n)$.
    \item Next, the decoder tries to determine $e_{a, n + 1}$ and $e_{b, n + 1}$. To do so, they search for the least index $e$ such that $\Phi_e(f(x)^{(\alpha)})$ is total and agrees with $f(a)$ on the first $m$ bits and concludes that this index $e$ is $e_{a, n + 1}$. They then find $e_{b, n + 1}$ in a similar manner. Note that since the decoder has access to $f(x)^{(\alpha + 2)}$, they can determine which programs are total when given $f(x)^{(\alpha)}$ as an oracle.
    \item Finally, the decoder tries to decode the $(n + 1)^\text{th}$ bit of $x$. To do so, they simply check which of the values $e_{a, n + 1}$ and $e_{b, n + 1}$ is larger than the other.
\end{enumerate}
In other words, the decoder determines $e_{a, n + 1}$ and $e_{b, n + 1}$ by finding the first programs which agree with $f(a)$ and $f(b)$ on certain long initial segments (and are total). To decide which initial segments of $f(a)$ and $f(b)$ to use, the decoder uses $f(c)$ and $e_{c, n}$.

\medskip\noindent\textbf{Coding process.} 
We will now explain what the coder should do on each step. Suppose the coder has just finished stage $n$ and is at the beginning of stage $n + 1$. For concreteness, we will again assume that $c$ plays the helper role on stage $n + 1$. The coder should do the following.
\begin{enumerate}
    \item First, the coder encodes the next bit of $x$ by picking new targets $a_{n + 1}$ and $b_{n + 1}$ for $a$ and $b$. For example, if the coder wants to arrange that $e_{a, n + 1} < e_{b, n + 1}$ then they may simply let $a_{n + 1} = a_n$ (so $e_{a, n} = e_{a, n + 1}$) and pick some $b_{n + 1} \in [T]$ which is Turing equivalent to $x$, has $B_n$ as an initial segment and for which the index of the least program computing $f(b_{n + 1})$ from $f(x)^{(\alpha)}$ is larger than $e_{a, n}$. Such a $b_{n + 1}$ must exist since the tree $T$ has infinitely many paths which are Turing equivalent to $x$ and which begin with $B_n$ and since $f$ is injective on $[T]$. The case in which the coder wants to arrange that $e_{a, n + 1} > e_{b, n + 1}$ is similar.
    \item Next, the coder finds a number $m$ large enough that for every $e < e_{a, n + 1}$, $\Phi_e(f(x)^{(\alpha)})$ is either not total or disagrees with $f(a_{n + 1})$ below $m$ and similarly for $e_{b, n + 1}$ and $f(b_{n + 1})$. In other words, looking only at the length $m$ initial segments of $f(a_{n + 1})$ and $f(b_{n + 1})$ is enough to determine $e_{a, n + 1}$ and $e_{b, n + 1}$. The point is that $m$ is large enough that the decoder only has to check the length $m$ initial segments of $f(a)$ and $f(b)$ to correctly determine $e_{a, n + 1}$ and $e_{b, n + 1}$.
    \item The coder now picks a new target $c_{n + 1}$ for $c$ such that $f(c_{n + 1})$ disagrees with $f(c_n)$, but the first point of disagreement is larger than $m$. To see why this is possible, note that since $f$ is continuous on $[T]$, we can take a long enough initial segment of $c_n$ so that for any real $ y \in [T]$ which agrees with $c_n$ on that initial segment, $f(y)$ and $f(c_n)$ will agree on their first $m$ bits. Furthermore, $[T]$ contains infinitely many paths which extend this initial segment and which are Turing equivalent to $x$. Since $f$ is injective on $[T]$, there must be some such path $y$ such that $f(y)$ disagrees with $f(c_n)$. We can take $c_{n + 1}$ to any such $y$.
    \item Let $m'$ be the first point of disagreement between $f(c_n)$ and $f(c_{n + 1})$. The coder now takes $C_{n + 1}$ to be a long enough initial segment of $c_{n + 1}$ to ensure that $f(c)$ must agree with $f(c_{n + 1})$ on the first $m'$ bits. To see why this is possible, simply note that $f$ is continuous on $[T]$.
    \item The coder now takes $A_{n + 1}$ and $B_{n + 1}$ to be initial segments of $a_{n + 1}$ and $b_{n + 1}$ long enough to ensure that $f(a)$ and $f(b)$ agree with $f(a_{n + 1})$ and $f(b_{n + 1})$ on the first $m'$ bits.
\end{enumerate}
Note that to carry out this coding process, we just need to be able to compute $f(x)^{(\alpha + 2)}$ (the $+ 2$ is needed to figure out which programs are total) and to compute $f$. Since $f$ is computable and $\alpha < \omega_1^x$, the whole process is therefore hyperarithmetic in $x$. Furthermore, it is easy to check that if the coder works as described above then the decoder will correctly determine the necessary indices on each step.
\end{proof}

\section{The Case of Borel Functions}
\label{section-borel}

It is popular to suppose that any proof of Martin's conjecture will only use determinacy in a ``local'' way---that is, the proof will still work in $\ZF$ when restricted to Borel functions, just by replacing the original uses of $\AD$ with analogous uses of Borel determinacy.

In this section, we will see that the main result of this paper does hold in $\ZF$ when restricted to Borel functions, but that proving this requires using a trick not present in the $\AD$ proof presented above. The trouble is that even if we only consider Borel functions, the proof of Lemma \ref{lemma-replacef} appears to require analytic determinacy rather than Borel determinacy. However, this can be avoided by a more careful analysis and an appeal to $\Sigma^1_1$-bounding.

Here's the key idea. If $f$ is hyp-regressive then we know that for each $x$ there is some $\alpha < \omega_1^x$ such that $x^{(\alpha)}$ computes $f(x)$. We will use $\Sigma^1_1$-bounding to find a single $\alpha$ which works for all $x$. After this, it will be straightforward to modify the proof of Lemma \ref{lemma-replacef} to only use Borel determinacy. 

In the next lemma we will prove this key point. Note that since we are restricting ourselves to Borel functons, we can drop the ``hyp-regressive'' requirement---every Borel function $f$ is automatically regressive on a cone of hyperarithmetic degrees.

\begin{lemma}[$\ZF$]
\label{lemma-boundalpha}
Let $f \colon 2^\omega \to 2^\omega$ be a Borel function. Then there is some $\alpha < \omega_1$ such that for all $x$ on a cone of hyperdegrees, $\alpha < \omega_1^x$ and $x^{(\alpha)} \ge_T f(x)$.
\end{lemma}

\begin{proof}
As noted above, since $f$ is Borel, $f(x) \le_H x$ on a cone of hyperdegrees. For the rest of the proof, we will implicitly work on this cone and thus we may assume $f(x) \le_H x$ for all $x$.

We start by simply writing down the definition of hyperarithmetic reducibility: for each $x$, we know that $f(x) \le_H x$ and hence that there is some $\alpha < \omega_1^x$ such that $x^{(\alpha)}$ computes $f(x)$. Our goal is to show that there is some $\alpha < \omega_1$ which is large enough to work for all $x$. We will do so by using $\Sigma^1_1$-bounding.

Let $A$ be the set of reals $r$ which code presentations of linear orders such that for some $x$,
\begin{enumerate}
    \item $x$ computes $r$
    \item $r$ has no infinite descending sequences which are hyperarithmetic in $x$
    \item and there is a jump hierarchy $H$ on $r$ starting from $x$ such that $H$ does not compute $f(x)$.
\end{enumerate}
By Lemma \ref{lemma-harrisondefinable} plus the fact that $f$ is Borel, the set $A$ is $\mathbf{\Sigma^1_1}$ definable (note that this is boldface rather than lightface because $f$ is Borel but not necessarily lightface $\Delta^1_1$).

Next, I claim that $A$ only contains well-orders. Suppose not and that $A$ contains an ill-founded order, $r$. Let $x$ witness that $r$ is in $A$. Then $r$ is a pseudo-well-order relative to $x$. But by Lemma \ref{lemma-pseudowellordercomputes}, this means that any jump hierarchy on $r$ starting with $x$ computes everything hyperarithmetic in $x$, and in particular, computes $f(x)$. This contradicts the definition of $A$.

Since $A$ is $\mathbf{\Sigma^1_1}$ and contains only well-orders, $\mathbf{\Sigma^1_1}$-bounding implies that there is some $\alpha < \omega_1$ which bounds everything in $A$. By the definition of $A$ this means that for every $x$ either $\omega_1^x \le \alpha$ or $x^{(\alpha + 1)} \ge_T f(x)$. So if we go to a cone on which everything computes a presentation of $\alpha$ then we obtain the conclusion of the lemma.
\end{proof}

We can now prove the Borel version of Theorem \ref{thm-main}.

\begin{theorem}[$\ZF$]
Let $f \colon 2^\omega \to 2^\omega$ be a hyp-invariant Borel function. Then either $f$ is constant on a cone of hyperdegrees or $f(x) \ge_H x$ on a cone of hyperdegrees.
\end{theorem}

\begin{proof}
By the previous lemma, we can assume there is some $\alpha < \omega_1$ such that for all $x$ on a cone of hyperdegrees, $\alpha < \omega_1^x$ and $x^{(\alpha)} \ge_T f(x)$. Let $a$ be the base of such a cone and let $r$ be a presentation of $\alpha$ computable from $a$. For the rest of the proof, we will work on the cone above $a$ and we will interpret $x^{(\alpha)}$ to mean the unique jump hierarchy on $r$ that starts with $x$.

The main idea of the proof is to go through the proof of Theorem \ref{thm-main} and make sure that every time that proof used determinacy, we can actually get by with just Borel determinacy. The only part of that proof in which we used determinacy was in the proof of Lemma \ref{lemma-replacef}. In particular, we used determinacy by applying Lemma \ref{lemma-computableuniformization} to the binary relation $R$ defined by
\[
R(x, y) \iff x \ge_T y \text{ and } f(x) \equiv_H y.
\]

The problem is that even if $f$ is Borel, this relation is not $\mathbf{\Delta^1_1}$, but only $\mathbf{\Pi^1_1}$ (since, in general, the formula $x \equiv_H y$ is only $\Pi^1_1$). We will remedy this problem by showing that the relation $R$ can be replaced by the relation $S$ defined by
\[
S(x, z) \iff x \ge_T z \text{ and } \exists y \le_T x\, (y \ge_T a \land x \le_T y^{(\alpha)} \land f(y) = z).
\]
In particular, we will show that the domain of $S$ is cofinal in the Turing degrees. The requirement that $y$ must compute $a$ is necessary to ensure that $y^{(\alpha)}$ is well-defined (and note that it implies that $x$ must also compute $a$).

\medskip\noindent\textbf{Why is this sufficient?} Let's first assume that we can show that the domain of $S$ is cofinal and see why that is enough to complete the proof. Since the definition of $S$ is $\mathbf{\Delta^1_1}$, and satisfies the conditions of Lemma \ref{lemma-computableuniformization}, there is a pointed perfect tree $T$ and a Turing functional $\Phi$ such that for all $x \in [T]$, $S(x, \Phi(x))$ holds. 

We now claim that $\Phi(x) \equiv_H f(x)$. To see why, let $y$ be a witness to the truth of $S(x, \Phi(x))$. Then $y \geq_T a$ and so $\alpha < \omega_1^y$. Also $y \leq_T x$ and $x \leq_T y^{(\alpha)}$, hence $x$ and $y$ are hyp-equivalent. Since $f$ is hyp-invariant, this implies that $\Phi(x) = f(y) \equiv_H f(x)$.

Thus we have recovered the conclusion of Lemma \ref{lemma-replacef} and the rest of the proof works unchanged. 

\medskip\noindent\textbf{Why is this true?} Now we will show that $S$ has cofinal domain. The proof is very similar to the proof of lemma \ref{lemma-replacef}. Let $x$ be any real. By joining with $a$ if necessary, we may assume that $x$ is in the cone above $a$. Since $x$ is in the cone above $a$, we know that $f(x) \le_T x^{(\alpha)}$ and so $S(x^{(\alpha)}, f(x))$ holds (as witnessed by $x$ itself). Since $x \le_T x^{(\alpha)}$, we have succeeded in finding something in the domain of $S$ which is above $x$.
\end{proof}

\bibliographystyle{plain}
\bibliography{bibliography}

\end{document}